\documentclass[11pt,oneside]{article}
\usepackage[utf8]{inputenc}
\usepackage[english]{babel}
\usepackage{amsmath}
\usepackage{amsfonts}
\usepackage{amssymb}
\usepackage{amsthm}
\usepackage{tikz-cd} 
\usepackage[left=3cm,right=3cm,top=3cm,bottom=3cm]{geometry}
\usepackage[title]{appendix}
\usepackage[]{algorithm2e}

\usepackage[normalem]{ulem}
\useunder{\uline}{\ul}{}
\theoremstyle{plain}
\newtheorem{teo}{}[section]
\newtheorem{prop}[teo]{Proposition}

\newtheorem{lem}[teo]{Lemma}
\newtheorem{thm}[teo]{Theorem}
\theoremstyle{definition}
\newtheorem{ex}[teo]{Example}
\newtheorem{rem}[teo]{Remark}
\newtheorem{df}[teo]{Definition}

\newcommand\blfootnote[1]{%
  \begingroup
  \renewcommand\thefootnote{}\footnote{#1}%
  \addtocounter{footnote}{-1}%
  \endgroup
}

\title{On the dynamics of the combinatorial model of the real line}
\author{Pedro J. Chocano}
\date{}

\begin{document}
\maketitle

\begin{abstract}
We study dynamical systems defined on the combinatorial model of the real line. We prove that using single-valued maps there are no periodic points of period $3$, which contrasts with the classical and less restrictive setting. Then, we use Vietoris-like multivalued maps to show that there is more flexibility, at least in terms of periods, in this combinatorial framework than in the usual one because we do not have the conditions about the existence of periods given by the Sharkovski Theorem.
\end{abstract}

\section{Introduction}\label{sec:introduccion}
\blfootnote{2020  Mathematics  Subject  Classification:   06A07,  	37B02,  	37E15.}
\blfootnote{Keywords: Posets, dynamical systems, periodic points, combinatorial model real line.}
\blfootnote{This research is partially supported by Grant PGC2018-098321-B-100 from Ministerio de Ciencia, Innovación y Universidades (Spain).}

In the recent years, there has been an increasing interest in the theory of finite spaces and Alexandroff spaces and their application in computational dynamical systems and other areas (see \cite{barmak2020Lefschetz,mrozek2019persistent,lipinski2019conley,chocano2020topological,chocano2022onsome} and references therein). Another approach uses simplicial complexes or Lefschetz complexes, that are viewed as combinatorial objects, see \cite{mrozek2017conley} or \cite{mrozek2021combinatorial}. In general, the idea of these methods is to construct dynamical systems on combinatorial objects from classical dynamical systems (such as flows) and then study some of the dynamical information there. It is worth pointing out that the inverse problem (construct from a combinatorial dynamical system a flow or semiflow in the geometric realization of a simplicial complex preserving dynamical properties) has been discussed in \cite{mrozek2021semiflows} with positive results.

There is also a different approach that reconstructs or approximates discrete dynamical systems by sequences of combinatorial dynamics defined on finite spaces \cite{chocano2020coincidence}. This idea arises from the good properties of reconstruction for compact metric spaces that have some inverse sequences of finite topological spaces (see \cite{clader2009inverse} and \cite{chocano2021computational}). The key notion is to consider a special class of multivalued maps, the so-called Vietoris-like multivalued maps, that generalizes continuous single-valued maps and the multivalued maps considered in \cite{barmak2020Lefschetz}. One of the main reasons to consider multivalued maps is the following result:
\begin{thm}[\cite{chocano2022ontriviality}] Let $A$ be an Alexandroff space and let $\varphi:\mathbb{R}\times A\rightarrow A$ be a flow. Then $\varphi$ is trivial, i.e., $\varphi(t,x)=x$ for any $t\in \mathbb{R}$ and $x\in A$.
\end{thm}
Moreover, if $A$ is finite we get that for every homeomorphism $f:A\rightarrow A$ there exists a positive integer number $n_f$ such that $f^{n_f}$ is the identity map. Hence, single-valued maps are not suitable to define or adapt the usual definition of flow for Alexandroff spaces or the usual definition of discrete dynamical system for finite spaces. 

The goal of this work is to show that single-valued maps are not suitable to study dynamical systems also for non-finite Alexandroff spaces and that there are somehow more richness of dynamics in this combinatorial framework, using Vietoris-like multivalued maps, than in the classical one. To do this we consider the combinatorial model of the real line and study the dynamics generated by single-valued maps and Vietoris-like multivalued maps having in mind a classical result about periodic points. Recall that for the real line $\mathbb{R}$, there is the beautiful theorem of Sharkovski (for a recent account of this result and its history we refer the reader to \cite{burns2011Sharkovski} and the references given there): 
\begin{thm}\label{thm:sharkovskiTheorem} Consider the total ordering of $\mathbb{N}$ given by 
$$3\triangleright 5\triangleright 7 \triangleright \cdots \triangleright 2\cdot 3 \triangleright 2\cdot 5 \triangleright 2\cdot 7 \triangleright \cdots \triangleright 2^2\cdot 3\triangleright 2^2 \cdot 5 \triangleright 2^2 \cdot 7\triangleright \cdots \triangleright 2^3 \triangleright 2^2 \triangleright 2 \triangleright 1.$$
Suppose $f:\mathbb{R}\rightarrow \mathbb{R}$ is a continuous map that has a periodic point of period $m$ and $m\triangleright l$. Then $f$ also has a periodic point of period $l$. 
\end{thm}

The organization of the paper is as follows. In Section \ref{sec:preliminaries} we review some of the standard facts on the theory of Alexandroff spaces and fix notation. Section \ref{sec:dynamicaUnievaluada} is devoted to the study of the dynamics generated by continuous single-valued maps defined on the combinatorial model of the real line. To be specific, we show that there are only periodic points of period $2$ and classify the discrete dynamics. This contrasts with Theorem \ref{thm:sharkovskiTheorem} and proves that single-valued maps are not useful to model classical dynamics of the real line. Finally, in Section \ref{sec:dinamicaMultievaluada} we introduce Vietoris-like multivalued maps and provide several examples that prove the variety of situations that may arise when we consider Vietoris-like multivalued maps. For example, there is a Vietoris-like multivalued map having periodic points of every period. But we also give an example where there are only periodic points of period $3$. Consequently, we do not have the restrictions that Theorem \ref{thm:sharkovskiTheorem} gives and we might expect to get more richness in terms of periodic points.


\section{Preliminaries}\label{sec:preliminaries}
We fix notation and recall basic theory of Alexandroff spaces. The results of this section can be found in \cite{barmak2011algebraic} and \cite{may1966finite}.

\begin{df} An Alexandroff space is a topological space for which arbitrary intersections of open sets are still open.
\end{df}
Given an Alexandroff space $X$ and $x\in X$, $U_x$ is the open set given by the intersection of every open set containing $x$. Consider $x,y\in X$ and define $x\leq y$ if and only if $U_x\subseteq U_y$. If $X$ is a $T_0$-space, then the previous relation is a partial order on $X$. We can also obtain an Alexandroff $T_0$-space from a partially ordered set $(X,\leq)$. Consider the upper sets of $(X,\leq)$ as a basis. These two relations are mutually inverse. From now on, we assume that every Alexandroff space is a $T_0$-space and treat Alexandroff spaces and partially ordered sets (posets) as the same object without explicit mention. 

A map $f:(X,\leq_X)\rightarrow (Y,\leq_Y)$ is order-preserving if $x\leq_X y$ implies $f(x)\leq_Y f(y)$. It is an easy exercise to show that a map $f:X\rightarrow Y$ between two Alexandroff spaces is order-preserving if and only if it is continuous. Therefore, we have:

\begin{thm} The category of Alexandroff $T_0$-spaces is isomorphic to the category of partially ordered sets.
\end{thm}

Let $X$ be an Alexandroff space and suppose that $U_x$ is finite for every $x\in X$. The Hasse diagram of $X$ is a digraph constructed as follows. The vertices are the points of $X$. There is an edge from $x$ to $y$ if and only if $x< y$ and there is no $z$ with $x<z<y$. In subsequent Hasse diagrams we assume an upward orientation. Moreover, a \textit{fence} in $X$ is a sequence $x_0,x_1,...,x_n$ of points such that any two consecutive points are comparable. The \textit{height} of $X$ is one less than the maximum number of elements in a chain of $X$. The height of a point $x\in X$ is defined by $ht(x)=ht(U_x)$.

We now see how Alexandroff spaces and simplicial complexes are related.

\begin{df} Let $X$ be an Alexandroff space. The order complex $\mathcal{K}(X)$ associated to $X$ is the simplicial complex whose simplices are the non-empty chains of $X$.
\end{df}
Let $f:X\rightarrow Y$ be a continuous map between posets. Then there is a natural and well-defined simplicial map $\mathcal{K}(f):\mathcal{K}(X)\rightarrow \mathcal{K}(Y)$. A simplex $\sigma$ in $\mathcal{K}(X)$ is given by a chain $x_1<...<x_n$. From this, $\mathcal{K}(f)(\sigma)$ is the simplex in $\mathcal{K}(Y)$ given by the chain $f(x_1)\leq...\leq f(x_n)$. 
\begin{df} Let $K$ be a simplicial complex. The face poset $\mathcal{X}(K)$ associated to $K$ is the poset of simplices of $K$ ordered by inclusion.
\end{df}
Similarly, if $f:K\rightarrow N$ is a simplicial map, then there is a natural continuous map $\mathcal{X}(f):\mathcal{X}(K)\rightarrow \mathcal{X}(N)$. We denote by $|K|$ the geometric realization of the simplicial complex $K$. Notice that these vertical bars $|\cdot|$ will also be used  to denote the cardinality of sets.  
\begin{thm}
For each Alexandroff space $X$ there exists a weak homotopy equivalence $f:|\mathcal{K}(X)|\rightarrow X$. For each simplicial complex $L$ there exists a weak homotopy equivalence $f:|L|\rightarrow \mathcal{X}(L)$.
\end{thm}
The face poset of a simplicial complex provides us a combinatorial model of the original topological space. Let us consider the real line $\mathbb{R}$ and a triangulation on it. Then we have that its combinatorial model is given by $L=\{x_i\}_{i\in \mathbb{Z}}$ with $x_i< x_{i+1}> x_{i+2}$ if $i$ is odd and $x_{-1}<x_0>x_{1}$ (see Figure \ref{fig:CombRealLine}).
\begin{figure}[h]
\centering
\includegraphics[scale=1]{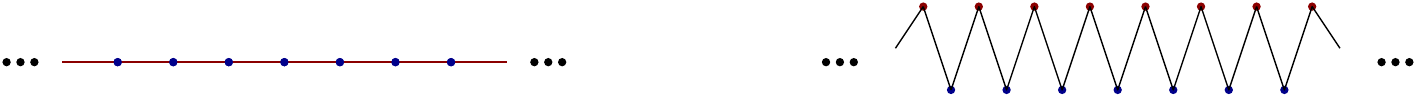}\caption{On the left a triangulation of the real line and on the right the Hasse diagram of its combinatorial model.}\label{fig:CombRealLine}
\end{figure}

In what follows, $L$ stands for the face poset of a triangulation given on the real line. Let $a,b\in L$. We say that $[a,b]$ is the set of points that are in the fence $a\geq...\leq b$ and we call it interval. We write $(a,b)=[a,b]\setminus\{a,b \}$. For simplicity, we will denote also $[a,b]=[b,a]$. For instance, $[x_3,x_0]=\{x_0,x_1,x_2,x_3 \}=[x_0,x_3]$ because $x_0>x_1<x_2>x_3$. We say that a sequence of points $\{t_j \}_{j\in \mathbb{N}}$ in $L$ tends to $x_{\infty}$ ($x_{-\infty}$) if for every $x_i\in L$ there exists $n_0\in \mathbb{N}$ satisfying that $t_n=x_{i_n}$ with $i_n>i$ ($i_n<i$) whenever $n\geq n_0$.  This condition means that the sequence goes to the right (left). Observe that for every pair of points $a$ and $b$ in $L$, $[a,b]$ is a finite topological space which is contractible. 

Given a continuous map $f:X\rightarrow X$, a point $x\in X$ is called a periodic point if there exists $m$ such that $f^{m}(x)=x$. We say that $x$ is a periodic point of period $n$ if $n$ is the smallest positive integer satisfying that $f^n(x)=x$ and $x$ is a fixed point if $n=1$.

\section{Dynamics of the combinatorial model of the real line}\label{sec:dynamicaUnievaluada}
In this section we describe the dynamics generated by a continuous map $f:L\rightarrow L$.

\begin{lem}\label{lem:IntervaloAintervalo} Let $f:L\rightarrow L$ be a continuous map and $a,b\in L$. Then $[f(a),f(b)]\subseteq f([a,b])$.
\end{lem}
\begin{proof}
By definition $[a,b]=a\leq x_1\geq x_2\leq ...\geq b$. By the continuity of $f$ we get
$$ f(a)\leq f(x_1)\geq f(x_2)\leq ...\geq f(b), $$
and so $[f(a),f(b)]\subseteq f([a,b])$.
\end{proof}

\begin{rem}\label{rem:cardinalBaja} Let $f:[a,b]\rightarrow L$ be a continuous map, where $a,b\in L$. Then  $|[a,b]|\geq |f([a,b])|\geq|[f(a),f(b)]|$.
\end{rem}
\begin{prop}\label{prop:noHayPeriodo3} Let $f:L\rightarrow L$ be a continuous map. Then $f$ does not have a point of period $3$.
\end{prop}
\begin{proof}
We argue by contradiction. Suppose that $f$ contains a periodic point of period $3$. Then there exists $x$ such that $x\neq f(x)$, $x\neq f^2(x)\neq f(x)$ and $f^3(x)=x$. Let us denote $f(x)=y$ and $f(y)=z$ for simplicity. By Remark \ref{rem:cardinalBaja}, we get
$$|[x,y]|\geq |f([x,y])|\geq |[y,z]|\geq |f([y,z])| \geq  |[z,x]|\geq |f([z,x])|\geq |[x,y]|, $$
which yields that $|[x,y]|= |f([x,y])|=|[y,z]|= |f([x,y])|=|[z,x]|=|f([z,x])|$, and using Lemma \ref{lem:IntervaloAintervalo} $f([x,y])=[y,z]$, $f([y,z])=[z,x]$ and $f[(z,x)]=[x,y]$. Therefore, $f_{|[x,y]}:[x,y]\rightarrow [y,z]$, $f_{|[y,z]}:[y,z]\rightarrow [z,x]$ and $f_{|[z,x]}:[z,x]\rightarrow [x,y]$ are homeomorphisms and so $ht(x)=ht(y)=ht(z)$. This means that one of these three points, let us say $y$, satisfies that $y\in [x,z]$. By the cardinality of each of the previous intervals we get that $x=y=z$, which entails the contradiction.
\end{proof}
Notice that the same arguments may be adapted to get
\begin{thm}\label{thm:dinamica_periodo_3_no} Let $f:L\rightarrow L$ be a continuous map. Then $f$ does not have periodic points of period $n\geq 3$.
\end{thm}

We now focus on the periodic points of period $2$. It is easy to find continuous maps having periodic points of period $2$ (consider $f(x_i)=x_{-i}$ and $f(x_0)=x_0$). We describe the continuous maps having such property.
\begin{thm}\label{thm:DinamicaPeriodo2} Let $f:L\rightarrow L$ be a continuous map that has a periodic point of period $2$. Then $f$ satisfies only one of the following properties.
\begin{itemize}
\item $f$ is a homeomorphism that fixes only one point $z$ and every $s\in L\setminus \{z\}$ is a periodic point of period $2$.
\item $f$ fixes only one point $z$ and there exists an interval $[x,y]$ containing $z$ such that  every $t\in [x,y]\setminus \{z\}$ is a periodic point of period $2$ and for every $t\in L\setminus [x,y]$ there exists a positive integer number $n$ with $f^n(t)\in [x,y]$. 
\end{itemize}
\end{thm}
\begin{proof}
By hypothesis there exists a point $x$ such that $f(x)=y\neq x$ and $f^2(x)=x$. Using Remark \ref{rem:cardinalBaja} and Proposition \ref{lem:IntervaloAintervalo} we get that $f_{|[x,y]}:[x,y]\rightarrow [x,y]$ is a homeomorphism. Hence, $ht(x)=ht(y)$. Without loss of generality we assume that $ht(x)=1$. On the other hand, by the Lefschetz fixed point theorem (see \cite{baclawski1979fixed}) there exists a fixed point, denoted by $z$, for $f_{|[x,y]}$. Again, we get that $|[x,z]|=|[z,y]|$ and $f_{|[x,z]}:[x,z]\rightarrow [z,y]$ is a homeomorphism by Remark \ref{rem:cardinalBaja} and Lemma \ref{lem:IntervaloAintervalo}. Therefore, $f_{|[x,y]}$ fixes $z$ and for every point $t\in(x,z)$,  $f(t)\in(z,y)$ and $f^2(t)=t$. 

It only remains to study the behaviour of $f$ outside $[x,y]$. Suppose that there exists a fixed point $t\notin [x,y]$. Repeating some of the arguments used before we may conclude that $f_{|[z,t]}:[z,t]\rightarrow [z,t]$ is a homeomorphism, which clearly contradicts the conclusion obtained above. Therefore, for every $s$ with $s\neq f(s)$ and $f^2(s)=s$, we have that $z\in[s,f(s)]$. 

Consider $\textnormal{P}(2)=\{ s\in L| s\neq f(s)$ and $s=f^2(s)\}\cup \{ z\}$. Then $\textnormal{P}(2)$ is an interval. If $|\textnormal{P}(2)|$ is non-finite, then $f$ is a homeomorphism satisfying the first property. We now study the case where $|\textnormal{P}(2)|$ is finite. Since it is finite, there exist $x,y\in \textnormal{P}(2)$ satisfying that $[x,y]=\textnormal{P}(2)$. We argue by contradiction, suppose that there exists $t\in L \setminus [x,y]$ such that $f^n(t)\notin [x,y]$ for every $n\in \mathbb{N}$. Without loss of generality, assume that $t$ is on the right of $y$. By the continuity of $f$, we get that $f(t)$ is on the left of $x$, $f^2(t)$ is on the right of $y$ and so on. By Remark \ref{rem:cardinalBaja},
$$|[x,t]|\geq |[x,f^2(t)]|\geq |[x,f^4(t)]|\geq...\geq |[x,f^{2n}(t)]|\geq ... $$
It is clear that the sequence $\{ |[x,f^{2n}(t)]|\}_{n\in \mathbb{N}}$ should stabilize, otherwise we get $f^{2n}(t)=x$ for some $n\in \mathbb{N}$. Then there exists $n_0\in \mathbb{N}$ such that for every $n\in \mathbb{N}$ with $n\geq n_0$, $|[x,f^{2n_0}(t)]|=|[x,f^{2n}(t)]|$. This means that $t$ is a fixed point for $f^{2n_0}$. But $z$ is also a fixed point for $f^{2n_0}$, which gives the contradiction with the previous arguments about the existance of two fixed point for a continuous map having periodic points of period $2$. 
\end{proof}

This theorem describes the dynamics generated by a continuous map $f:L\rightarrow L$ having periodic points of period $2$, for a schematic representation of it see Figure \ref{fig:DinamicaPeriodo2}. The first picture corresponds to the first property of Theorem \ref{thm:DinamicaPeriodo2}. The second picture corresponds with the second property, where we have an attractor (points in the orange region). 
\begin{figure}[h]
\centering
\includegraphics[scale=0.7]{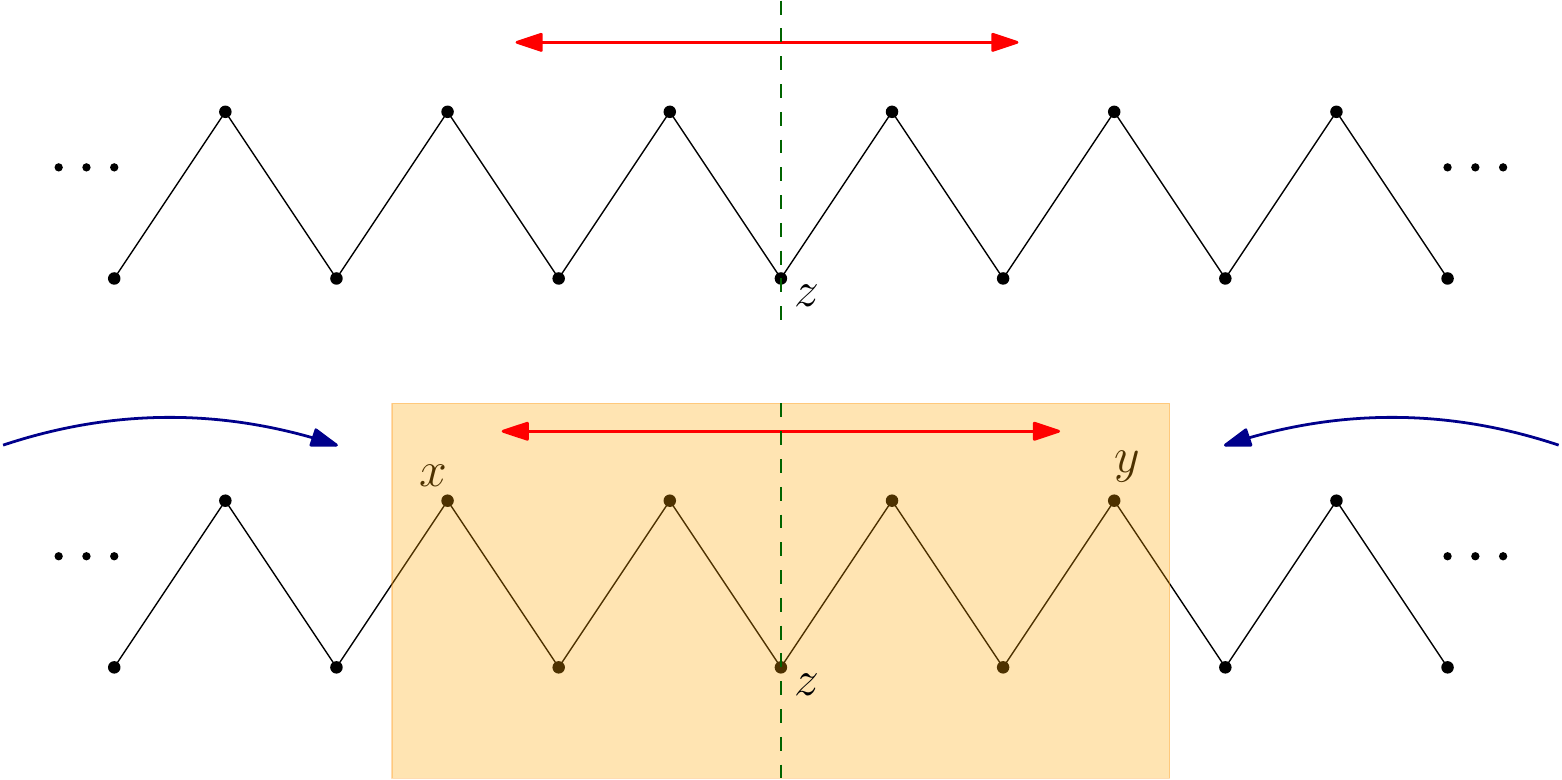}\caption{Schematic description of the situations described in Theorem \ref{thm:DinamicaPeriodo2}.}\label{fig:DinamicaPeriodo2}
\end{figure}
\begin{rem} We might say that the theorem of Sharkovski holds trivially in this setting since period $2$ implies that there exists a fixed point (Theorem \ref{thm:DinamicaPeriodo2} and $2\triangleright 1$) and there are no periodic points of period greater or equal than $3$ (Theorem \ref{thm:dinamica_periodo_3_no}). 
\end{rem}

To conclude this section, we study the behaviour of the dynamics generated by a continuous map without periodic points.
\begin{thm} Let $f:L\rightarrow L$ be a continuous map that does not have periodic points. Then $f$ satisifies only one of the following properties.
\begin{itemize}
\item $f$ is the identity map.
\item There exists an interval $[z,w]$ such that $f_{|[z,w]}$ is the identity map and for every $x\in L$, there exists a positive integer number $n$ satisfying that $f^{n}(x)\in [z,w]$.
\item $f$ does not have fixed points and for every $x\in L$ the sequence $\{f^n(x)\}_{n\in \mathbb{N}}$ tends to $x_{\infty}$ or for every $x\in L$ the sequence $\{f^n(x)\}_{n\in \mathbb{N}}$ tends to $x_{-\infty}$.
\end{itemize} 

\end{thm}
\begin{proof}
Consider the maximal interval for which $f$ is the identity map. If this interval is non-finite we get that $f$ is the identity. Suppose that this interval is finite and has at least two elements, we denote it by $[z,w]$. Without loss of generality assume that $ht(z)=1$. Consider $t<z$ such that $t\notin[ z,w]$. By the continuity of $f$ we have $f(t)\leq f(z)=z$ and since $U_z=\{ z,t,k\}$ for some $k<z$ we get that  $f(t)=t$ or $f(t)\in [z,w]$. The first condition contradicts the maximality of $[z,w]$. The second condition gives $f(t)=k$ or $f(t)=z$, but $z,k\in [z,w]$, since otherwise we get a contradiction with the fact that $f$ fixes the points of $[z,w]$ because $z$ would be an isolated fixed point. Repeating a similar argument we can deduce that for every $x\in L$ there exists a positive integer such that $f^n(x)\in [z,w]$. Suppose now that there exists only one fixed point $z$, without loss of generality assume that $ht(z)=1$. Consider $t,w<x$. If $f(t)=w$, then $f(w)=z$ since otherwise we have that $f(w)=t$, which implies that $t$ is a periodic point of period 2. Repeating previous arguments we can conclude that for every $l\in L$ there exists a positive integer number $n$ such that $f^n(l)=x
$.

Finally, assume that $f$ does not have fixed points. We argue by contradiction to get the third property. Suppose that there are two points $x_i,x_j\in L$ such that $f(x_i)=x_k$ with $k>i$ and $f(x_j)=x_m$ with $m<j$. By Lemma \ref{lem:IntervaloAintervalo}, $[x_m,x_k]\subseteq f([x_j,x_i])$. We now distinguish cases. If $[x_m,x_k]\subseteq [x_j,x_i]$, then $f_{|[x_j,x_i]}$ has a fixed point by the Lefschetz fixed point theorem (\cite{baclawski1979fixed}) and consequently we get a contradiction. If  $[x_j,x_i]\subsetneq [x_m,x_k]$, then  $|[x_j,x_i]|\geq |[x_m,x_k]|$ in contradiction with Remark \ref{rem:cardinalBaja}.
\end{proof}

The results about the dynamics of $L$ give us some restrictions about simplicial approximations. Concretely, we may not expect to get some dynamical information of a continuous map $f:\mathbb{R}\rightarrow \mathbb{R}$ from a simplicial approximation.  Let us consider a continuous map $f:\mathbb{R}\rightarrow \mathbb{R}$ with a periodic point of period $3$, for example, $f(x)=\min (1,1-2|x-\frac{1}{2}|)$. Suppose that $g$ is a simplicial approximation of $f$. Then it is clear that $g$ cannot have periodic points of period $n\geq 3$, otherwise we would get that $\mathcal{X}(g)$ is a continuous map having periodic points of period $n$.

\begin{thm} Let $f:\mathbb{R}\rightarrow \mathbb{R}$ be a continuous map and $g$ a simplicial approximation of $f$. Then $g$ can only have periodic points of period $2$ or fixed points.
\end{thm}

\section{Multivalued dynamics}\label{sec:dinamicaMultievaluada} 

We now introduce a sort of maps and multivalued maps that are very useful to define dynamics in Alexandroff spaces (one may refer to \cite{chocano2020coincidence} for a complete introduction and discussion about them). We recall that a topological space $X$ is acyclic if $\tilde{H}_i(X)=0$ for every $i\geq 0$, where $\tilde{H}_i(X)$ denotes the reduced homology group of $X$.
\begin{df} Let $f:X\rightarrow Y$ be a continuous map between two posets. We say that $f$ is a Vietoris-like map if for every chain $y_1<...<y_n$ in $Y$, $\bigcup_{i=1}^n f^{-1}(y_i)$ is acyclic.
\end{df}
Consider a multivalued map $F:X\multimap X$, where $X$ is a poset. The graph of $F$ is $\Gamma(F)=\{(x,y)\in X\times Y|y\in F(x) \}$, which is a subposet of $X\times Y$ with the product order. There are two natural projections $p:\Gamma(F)\rightarrow X$ and $q:\Gamma(F)\rightarrow Y$, that are continuous. 
\begin{df} Let $F:X\multimap Y$ be a multivalued map between two posets. We say that $F$ is a Vietoris-like multivalued map if $p$ is a Vietoris-like map.
\end{df}
Given a Vietoris-like multivalued map $F:X\rightarrow X$, a solution or orbit for $F$ is $\{t_i\}_{i\in \mathbb{Z}}$ such that $t_{i+1}\in F(t_{i})$ and  if $t_i=t_{i+1}\in F(t_i)$ and $|F(t_i)|\geq 2$ then there exists $n>i+1$ satisfying $t_{n}\neq t_{i+1}$. We say that $x$ is a fixed point if $x\in F(x)$  and $x_1$ is a periodic point of period $n$ if there exists a finite sequence of pairwise distinct points $\{x_1,...,x_{n}\}$ such that $x_2\in F(x_1)$, $x_3\in F(x_2)$, $...$, $x_1\in F(x_n)$. 

Vietoris-like maps are important because they induce morphisms in homology groups, that is, given a Vietoris-like map $F:X\multimap Y$, there is a natural morphism $F_*=H_*(X)\rightarrow H_*(Y)$. This leads us to apply homological methods to obtain dynamical information. A Lefschetz number $\Lambda(F)$ can be defined for $F$ if $X=Y$ (see \cite{chocano2020coincidence} for a detailed description). There is also a Lefschetz fixed point theorem in this setting.
\begin{thm}\label{thm:lefschetzfixed}
Let $X$ be a finite poset and let $F:X\multimap X$ be a Vietoris-like multivalued map such that $\Lambda(F)\neq 0$. Then there exists $x\in X$ such that $x\in F(x)$.
\end{thm}

Moreover, Vietoris-like maps play a fundamental role in methods of approximation to discrete dynamical systems. We explain it briefly: suppose that $X$ is a compact polyhedron. Then there exists an inverse sequence of finite spaces $(X_n,p_{n,n+1})$ satisfying that its inverse limit contains a homeomorphic of $X$ which is a strong deformation retract (\cite{clader2009inverse}). Let us consider now a continuous map $f:X\rightarrow X$, $f$ induces naturally a sequence of Vietoris-like multivalued maps $F_i:X_i\multimap X_i$. Studying the fixed points of $F_i$ we may locate the fixed points of $f$.

Regarding to our context, we give some examples to show the flexibility of the Vietoris-like multivalued maps to define dynamical systems for non-finite Alexandroff spaces. Firstly, we show that we can get a multivalued map with periodic points of every period.
\begin{ex} Consider $F:L\multimap L$ defined by 
\[F(x_i)=\begin{cases}
[x_0,x_2] & \text{if} \ i\leq 0 \\
[x_0,x_{i+1}] & \text{if} \ i> 0 \ \text{and} \ i \ \text{is odd} \\
[x_0,x_{i+2}] & \text{if} \ i> 0 \ \text{and} \ i \ \text{is even} .
\end{cases}
\]
$F$ is trivially a Vietoris-like multivalued map because $F$ is strongly upper semicontinuous with acyclic values (see \cite{barmak2020Lefschetz} and \cite{chocano2020coincidence}). For every $n\in \mathbb{N}$ there exists a periodic point of period $n$. Consider the sequence $\{ x_{i}\}_{i=1,...,n}$. By construction it is clear that $x_{i+1}\in F(x_i)$ where $1\leq i\leq n-1$ and $x_{1}\in F(x_n)$, which means that we have a periodic point of period $n$.
\end{ex} 
However, the following example shows that the Sharkovski theorem does not hold in this combinatorial framework. 
\begin{ex} Let $F:L\multimap L$ be defined by $F(x)=[x_1,x_2,x_3]$. It is easily seen that $F$ is a Vietoris-like multivalued map. We have $x_2\in F(x_1)$, $x_3\in F(x_2)$ and $x_1\in F(x_3)$, that is, a periodic point of period $3$. But there are no periodic points of period greater than $3$.
\end{ex} 
Note that this example is also telling us that there is more richness, at least in terms of periodic points, in the combinatorial setting than in the classical. Moreover, we cannot expect to get a result like the one obtained in \cite{mrozek2021semiflows}, where some combinatorial dynamics induce semiflows in the usual sense. To conclude, we provide an example of a multivalued map having different types of invariant sets at the same time (this only happens for the multivalued case).

\begin{ex}\label{ex:DinamicaMultievaluada} Consider $F:L\rightarrow L$ defined by 
\[ F(x_i)=\begin{cases}
x_{i+2} \ & \text{if} \ i<-7 \\
x_i \ & \text{if} \  -7\leq i\leq -5 \\
[x_{i},x_{i+1}] \ & \text{if} \ -5<i<-1\\
x_i  \ & \text{if} \ -1\leq i\leq 1 \\
[x_{i-1},x_{i}] \ & \ \text{if} \ 1< i<  5\\
x_i \ & \text{if} \ 5\leq i\leq 7 \\
[x_{i},x_{i+1}]  \ & \ \text{if} \ 7<i.
\end{cases}
\]
It is an easy exercise to prove that $F$ is a Vietoris-like multivalued map (because it is similar to a multiflow \cite{chocano2020coincidence}). There are three invariant sets: $[x_{-7},x_{-5}]$ (a saddle), $[x_{-1},x_{1}]$ (an attractor) and $[x_{5},x_{7}]$ (a repeller). For a schematic draw using the Hasse diagram of $L$ see Figure \ref{fig:DinamicaMultievaluada}.
\begin{figure}[h]
\centering
\includegraphics[scale=1.2]{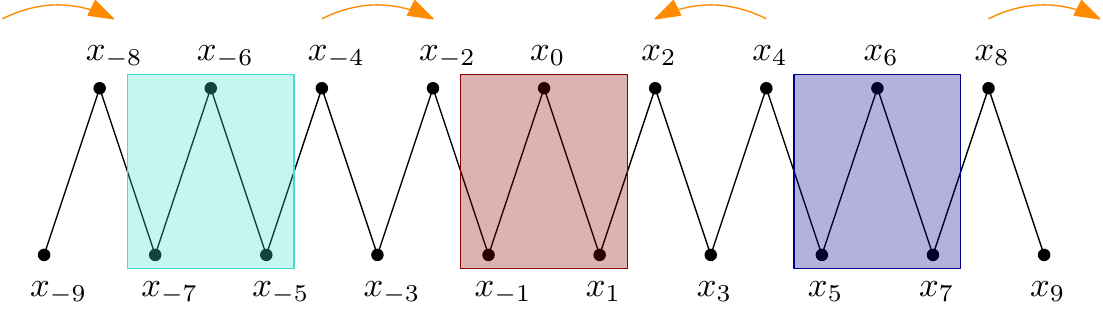}\caption{Schematic drawing of the dynamics of $F$.}\label{fig:DinamicaMultievaluada}
\end{figure}
\end{ex}
Notice that modifying the multivalued map $F$ of Example \ref{ex:DinamicaMultievaluada} it is very easy to get different situations. This example also points out that considering multivalued maps seems more suitable to speak about dynamics in this framework.

\textbf{Acknowledgement.} The author thanks Hector Barge for posing him the following question during DyToComp 2022: ``Is Sharkovski theorem true for Alexandroff spaces?''. This question has motivated the current paper. The author also thanks Jonathan Barmak for pointing out a mistake in Lemma 3.1 in a previous version of the paper.

\bibliography{bibliografia}

\begin{thebibliography}{10}

\bibitem{baclawski1979fixed}
K.~Baclawski and A.~Björner.
\newblock Fixed points in partially ordered sets.
\newblock {\em Adv. Math.}, 31(3):263--287, 1979.

\bibitem{barmak2011algebraic}
J.~A. Barmak.
\newblock {\em Algebraic topology of finite topological spaces and
  applications. \textnormal{Lecture Notes in Mathematics}}, volume 2032.
\newblock Springer, 2011.

\bibitem{barmak2020Lefschetz}
J.~A. Barmak, M.~Mrozek, and T.~Wanner.
\newblock A {L}efschetz fixed point theorem for multivalued maps of finite
  spaces.
\newblock {\em Math. Z.}, 294:1477--1497, 2020.

\bibitem{burns2011Sharkovski}
K.~Burns and B.~Hasselblatt.
\newblock The {S}harkovsky {T}heorem: {A} {Natural} {D}irect {P}roof.
\newblock {\em Am. Math. Mon.}, 18(3):229--244, 2011.

\bibitem{chocano2022ontriviality}
P.~J. Chocano, D.~Mond\'ejar, M.~A. Mor\'on, and F.~R. Ruiz~del Portal.
\newblock On the triviality of flows in {A}lexandroff spaces.
\newblock {\em Topol. App.}, (Accepted), 2022.

\bibitem{chocano2020coincidence}
P.~J. Chocano, M.~A. Mor\'on, and F.~R. Ruiz~del Portal.
\newblock Coincidence theorems for finite topological spaces.
\newblock {\em arXiv:2010.12804}, 2020.

\bibitem{chocano2020topological}
P.~J. Chocano, M.~A. Mor\'on, and F.~R. Ruiz~del Portal.
\newblock Topological realizations of groups in alexandroff spaces.
\newblock {\em Rev. R. Acad. Cien. Serie A. Mat.}, 115(25), 2021.

\bibitem{chocano2021computational}
P.~J. Chocano, M.~A. Mor\'on, and F.~R. Ruiz~del Portal.
\newblock Computational approximations of compact metric spaces.
\newblock {\em Phys. D}, 433(133168), 2022.

\bibitem{chocano2022onsome}
P.~J. Chocano, M.~A. Mor\'on, and F.~R. Ruiz~del Portal.
\newblock On some topological realizations of groups and homomorphisms.
\newblock {\em Trans. Amer. Math. Soc.}, 375:8635--8649, 2022.

\bibitem{clader2009inverse}
E.~Clader.
\newblock Inverse limits of finite topological spaces.
\newblock {\em Homol. Homotop. App.}, 11(2):223--227, 2009.

\bibitem{mrozek2019persistent}
T.~Dey, M.~Juda, T.~Kapela, J.~Kubica, M.~Lipiński, and M.~Mrozek.
\newblock Persistent {H}omology of {M}orse {D}ecompositions in {C}ombinatorial
  {D}ynamics.
\newblock {\em SIAM J. Appl. Dyn. Syst.}, 18(1):510--530, 2019.

\bibitem{lipinski2019conley}
M.~Lipiński, J.~Kubica, M.~Mrozek, and T.~Wanner.
\newblock Conley–{M}orse–{F}orman theory for generalized combinatorial
  multivector fields on finite topological spaces.
\newblock {\em Journal of Applied and Computational Topology}, pages 1--46,
  2022.

\bibitem{may1966finite}
J.~P. May.
\newblock Finite spaces and larger contexts.
\newblock {\em Unpublished book}, 2016.

\bibitem{mrozek2017conley}
M.~Mrozek.
\newblock Conley–{M}orse–{F}orman {T}heory for {C}ombinatorial
  {M}ultivector {F}ields on {L}efschetz {C}omplexes.
\newblock {\em Found. Comput. Math.}, (17):1585--1633, 2017.

\bibitem{mrozek2021combinatorial}
M.~Mrozek, R.~Srzednicki, J.~Thorpe, and T.~Wanner.
\newblock Combinatorial vs. classical dynamics: {R}ecurrence.
\newblock {\em Commun. Nonlinear Sci. Numer. Simul.}, 108:1--30, 2022.

\bibitem{mrozek2021semiflows}
M.~Mrozek and T.~Wanner.
\newblock Creating semiflows on simplicial complexes from combinatorial vector
  fields.
\newblock {\em J. Differ. equ.}, 304:375--434, 2021.

\end{thebibliography}
\bibliographystyle{plain}

\newcommand{\Addresses}{{
  \bigskip
  \footnotesize

  \textsc{ P.J. Chocano, Departamento de Matemática Aplicada,
Ciencia e Ingeniería de los Materiales y
Tecnología Electrónica, ESCET
Universidad Rey Juan Carlos, 28933
Móstoles (Madrid), Spain}\par\nopagebreak
  \textit{E-mail address}:\texttt{pedro.chocano@urjc.es}

}}

\Addresses

\end{document}